\documentclass[12pt,a4paper]{article}
\usepackage{amsmath, amsthm, amscd, amsfonts, amssymb, graphicx, color}
\usepackage{accents}
\usepackage[bookmarksnumbered, colorlinks, plainpages]{hyperref}
\usepackage{listings}
\makeatletter
\@namedef{subjclassname@2010}{%
	\textup{2010} Mathematics Subject Classification}
\makeatother
\usepackage[mathscr]{euscript}
\usepackage{latexsym}
\usepackage{multicol,tabularx}
\usepackage{tikz}
\usepackage{graphics}
\usepackage{mathtools}

\usepackage{amsfonts}
\usepackage{amssymb}
\newtheorem{theorem}{Theorem}
\usepackage[english]{babel}
\usepackage{multido}
\usepackage[nomessages]{fp}
\usepackage[margin=2.5cm]{geometry}
\usepackage{enumitem}
\newtheorem{lemma}{Lemma}
\newtheorem*{lemma*}{Lemma}

\newtheorem*{coro*}{Corollary}
\theoremstyle{definition}

\newtheorem*{defn*}{Definition}

\newtheorem*{res*}{Theorem}
\numberwithin{equation}{section}

\lstdefinestyle{mystyle}{
    basicstyle=\ttfamily\tiny,
    breakatwhitespace=false,
    breaklines=true,
}

\begin{document}

	\title{Erdős Conjecture and $AR$-Labeling}

 \date{}

\author{Arun J Manattu\footnote{E-mail: arunjmanattu@gmail.com} \ and
Aparna Lakshmanan S\footnote{E-mail: aparnals@cusat.ac.in, aparnaren@gmail.com}\\ Department of Mathematics\\
	Cochin University of Science and Technology\\Cochin - 682022, Kerala, India}
\maketitle
\begin{abstract}
 Given an edge labeling $f$ of a graph $G$, a vertex $v$ is called an $AR$-vertex, if $v$ has distinct edge weight sums for each distinct subset of edges incident on $v$. An injective edge labeling $f$ of a graph $G$ is called an $AR$-labeling of $G$, if $f:E(G) \rightarrow \mathbb{N}$ is such that every vertex in $G$ is an $AR$-vertex under $f$. The minimum $k$ such that there exists an $AR$-labeling $f:E\rightarrow \{1,2,3,\dots,k\}$ is called the $AR$-index of G, denoted by $ARI(G)$. In this paper, using a sequence originating from Erdős subset sum conjecture, a lower bound has been obtained for the $AR$-index of a graph and this bound is used to prove that only finitely many bistars, complete graphs and complete bipartite graphs are $AR$-graphs. The exact values of $AR$-index is obtained for stars and wheels. 
                        
\noindent\line(1,0){395}

\noindent\line(1,0){395}
\end{abstract}
    
\section{Introduction}
The Ramsey Theory is a branch of combinatorics that argues - in some sense philosophically - that absolute chaos is impossible in any system. The Schur's theorem as well as its generalization - the Rado's theorem \cite{Rob} were both precursors of Ramsey theory having the same essence. Motivated from Rado's partition regularity condition, an edge labeling of graphs called $AR$-labeling was introduced in \cite{Apa}. Given an edge labeled graph, a vertex $v$ is called an $AR$-vertex, if $v$ has distinct edge weight sums for each distinct subset of edges incident on $v$. i.e., if $\{x_1,x_2,\dots,x_k\}$ are the labels assigned to the edges incident on $v$, then the $2^k$ subset sums are all distinct. An injective edge labeling $f$ of a graph $G$ is said to be an $AR$-labeling of $G$ if $f:E(G) \rightarrow \mathbb{N}$ is such that every vertex in $G$ is an $AR$-vertex under $f$.  A graph $G$ is said to be an $AR$-graph, if there exists an $AR$-labeling $f:E \rightarrow \{1,2,3,\dots,m\}$, where $m$ denotes the number of edges in $G$. (We have used \cite{1} for graph theoretic terminology.) \par
The $AR$-labeling of a graph $G$ makes every subset of edges incident on each vertex unique up to the point of identifying any given subset of adjacent edges using an ordered pair $(u, k)$, where $u$ is a vertex in $G$ and $k \in \mathbb{N}$. If there is a connected graph representing a communication network with vertices having only local knowledge, $AR$-labeling provides the initiator of communications (server) with distinct commands for each distinct communication that could be translated only by those vertices receiving the command. If the server is assigning labels to the edges, the command, even though hacked, cannot be decrypted unless and until the individual vertices are compromised. So, $AR$-labeling has potential applications in security networks as well as defense systems.\par
In an information-theoretic interpretation  \cite{2}, namely in a setting of signaling over a channel with multiple access, we can identify the integers as pulse amplitudes that $n$ transmitters could transmit over an additive channel sending one bit of information each, for signaling the base station that they need to start a communication session. The requirement that all subset sums are distinct corresponds to the preference that the base station is able to deduce any possible collection of active users among the whole set.\par
Though $AR$-labeling has been introduced as an edge labeling, we can identify it as a vertex labeling with some restrictions. Let $S(\mathbb{N})$ denotes all possible subsets of natural numbers with distinct subset sums. A vertex labeling $g$ of $G$ is said to be an $AR$-labeling, if $g:V(G) \rightarrow S(\mathbb{N})$ is such that the images of every pair of adjacent vertices have exactly one element in common and any $n \in \mathbb{N}$ either do not appear in any set in the image or appear exactly twice in the sets in the image. A graph is said to be an $AR$-graph if there exists an $AR$-labeling $g:V(G) \rightarrow S(\{1,2,\dots,m(G)\})$ where $S(\{1,2,\dots,m(G)\})$ denotes all subsets of the set of first $m$ natural numbers having distinct subset sums. In fact, every edge labeling problem can be formulated as an equivalent vertex labeling problem. In the figures, we have used vertex labeling to represent the $AR$-labeling, since it helps the readers to identify the label assigned to each edge more clearly. 

\section{Erdős Subset Sum Conjecture and $ES$-sequence}

Let $\{a_1, a_2, \dots , a_n\}$ be a set of positive integers with $a_1 < a_2 <\dots < a_n$ such that all $2^n$ subset sums are distinct. A famous conjecture by Paul Erdős in 1931 states that $a_n > c\cdot 2^n$, for some constant c. Since the sequence arising from this conjecture is repeatedly used in this paper, we call it as the $ES$-sequence. The $n^{th}$ element of $ES$-sequence, $ES(n)$ denotes the smallest integer $m$ such that there exists a set of $n$ natural numbers $\{a_1, a_2, \dots, a_{n-1}, a_n = m\}, a_i < a_j$, for every $i < j$, such that all $2^n$ subset sums are distinct. This sequence appears in The On-Line Encyclopedia of Integer Sequences (OEIS) numbered A276661. Only the first nine numbers of this sequence are known and they are $1,2,4,7,13,24,44,84$ and $161$.\par
In 1955, using the second moment method \cite{Alo}, Erdős and Moser  \cite{3}
proved that $ES(n) \geq 2^n / (4\sqrt{n})$. No advances have been made so far in removing the term $(1/\sqrt{n})$ from this lower bound, but there have been several improvements on the constant factor including the work of Dubroff, Fox and Hu \cite{4}, Guy \cite{5}, Elkies \cite{6}, Bae \cite{7}, Aliev \cite{8} and Steinerberger \cite{Ste} while the best result known to date is still of the form $ES(n) > c \cdot 2^n / \sqrt{n}$. \par
In 1967, John Conway and Richard Guy \cite{9}
constructed a sequence of sets of integers which is now referred to as the Conway-Guy sequence. They showed that the first 40 sets of the Conway-Guy sequence have distinct subset sums and conjectured that all sets arising from their construction have distinct subset sums and are close to the best possible (with respect to the largest element). The first non-trivial upper bound of $ES(n)$ was hence given to be $2^{n-2}$ (for sufficiently large $n$). The $21^{st}$ set in the Conway-Guy sequence has largest element less than $2^{19}$. This gives the bound of $2^{n-2}$ for all $n > 21$ since from a given set of $n$ elements having distinct subset sums, we could construct a set of $(n+1)$ elements having distinct subset sums by doubling the $n$ elements of the first set and introducing an odd number to that set as the $(n+1)^{th}$ element.\par
In 1988, Fred Lunnon \cite{Lun} conducted an extensive computational investigation of
this problem and he determined that $ES(n)$ is given by the Conway-Guy sequence, for $n \leq 8$ and verified that the Conway-Guy sequence has distinct subset sums, for $n \leq 79$. Lunnon also gave a set of 67 integers which surpassed the improvements offered by Conway-Guy sequence in terms of the upper bound for $ES$-sequence.
In 1996, Tom Bohman proved that all sets arising from the Conway-Guy sequence have distinct subset sums \cite{10}. Bohman also gave an improvement to the upper bound of $ES$-sequence by introducing microscopic variations to the construction of Lunnon \cite{11}.
\section{AR-Index of Graphs}
Though there are infinitely many non $AR$-graphs, a theorem in the concluding remarks of \cite{Apa} states that given an arbitrary graph $G$, an $AR$-labeling of $G$ always exists, though the image set contains numbers greater than $m(G)$. An immediate question would be to find the smallest possible $k$ such that an $AR$-labeling exists from the set of edges to the first $k$ natural numbers and so was $AR$-index defined in \cite{Apa}. 
\begin{defn*}[$AR$-Index of G]
    The minimum $k$ such that there exists an $AR$-labeling $f:E\rightarrow \{1,2,3,\dots,k\}$ is called the $AR$-index of G, denoted by $ARI(G)$.
\end{defn*}
 The $AR$-index of a graph $G$ measures how close a graph is towards being an $AR$-graph and evidently $G$ is an $AR$-graph if $ARI(G) = m(G)$. Similarly, we can consider a graph as an almost $AR$-graph if $ARI(G) = m(G) + 1$. The $ES$-sequence that follows from the Erdős subset sum conjecture gives the following bounds to the $AR$-Index of a graph.

\begin{theorem} \label{Delta}
    For any graph $G$,  $ES(\Delta(G)) \leq ARI (G) \leq ES(m(G))$, where $\Delta(G)$ is the maximum degree of a vertex in $G$.
\end{theorem}

\begin{proof}
    The lower bound follows from the fact that there exists some vertex $v \in G$ that has degree $\Delta(G)$ and to make $v$ an $AR$-vertex, we have to use edge labels at least as large as $ES(\Delta)$.\par
    To prove the upper bound,  let $ES(m) = x$. This implies there exists an $m$-element set having maximum element $x$ with distinct subset sums. Labeling the edges of $G$ using the numbers in this set yields an $AR$-labeling of $G$. Hence, the result.  
\end{proof}

\begin{coro*}
    The $AR$-index of a star, $ARI(K_{1,n}) = ES(n)$. Moreover, a star is not an $AR$-graph, for $n > 2$.
\end{coro*}
\begin{proof}
    Since $\Delta(K_{1,n}) = m(K_{1,n}) = n$, the lower and upper bounds in Theorem \ref{Delta} coincide and hence $ARI(K_{1,n}) = ES(n)$. Also, $ES(n) > n$, for $n > 2$ and hence $K_{1,n}$ is not an $AR$-graph, for $n > 2$.
\end{proof}

\begin{theorem}
    Every graph $G$ can be identified as an induced subgraph of some $AR$-graph.
    \end{theorem}
    
\begin{proof}
    If $G$ is itself an $AR$-graph, then there is nothing to prove. Therefore, assume that $G$ is not an $AR$-graph. Let $G'$ be the graph obtained from $G$ by attaching a pendent vertex $v$ to any one of the vertices of $G$. If $G'$ is an $AR$-graph, then we are done. Otherwise, $l = ARI(G') - m(G') > 0$. Attach a path on $l$ vertices to $v$ to get $H$, so that we have $l$ new edges. Label those edges using $ARI(G') - m(G')$ labels that are not used in the edge labeling of $G'$. This is an $AR$-labeling of $H$ using labels from the set $\{1,2,\ldots,m(H)\}$ and hence $H$ is an $AR$-graph for which $G$ is an induced subgraph.
\end{proof}
\begin{coro*}
    The property of being an $AR$-graph is not vertex hereditary and hence $AR$-graphs do not admit forbidden subgraph characterization.
\end{coro*}

\begin{theorem}
    A bistar graph $B_{n,n}$ is an $AR$-graph if and only if $n \leq 2$ and $B_{3,3}$ is an almost $AR$-graph.
\end{theorem}
 
\begin{proof}
\begin{figure}[h]
       \centering
       \includegraphics[scale=0.75]{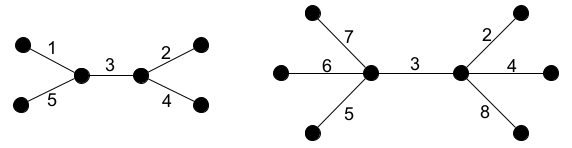}
       \caption{$AR$-Labeling of $B_{2,2}$ and $B_{3,3}$}
       \label{bistar}
   \end{figure}
Since, $\Delta(B_{n,n}) = n+1$, by Theorem \ref{Delta}, we have $ARI(B_{n,n}) \geq ES(n+1)$. Also, 
$B_{n,n}$ has $2n + 1$ edges and $2n+1 < ES(n+1)$, for $n > 3$. Hence, $B_{n,n}$ is not an $AR$-graph for $n > 3$. For $n = 3$, $ES(4) = 7 = m(B_{3,3})$. But, there are two vertices of degree $4$ in $B_{3,3}$ and there is only one $4$-element set $\{3, 5, 6, 7\}$ with distinct subset sums, having maximum element less than or equal to 7. Hence, $B_{3,3}$ is not an $AR$-graph.\par
Now, $B_{1,1}$ is $P_4$ which is trivially an $AR$-graph and the edge labeling in Figure \ref{bistar} shows that $B_{2,2}$ is an $AR$-graph and $ARI(B_{3,3}) = 8 = m(B_{3,3})+1$, which proves that $B_{3,3}$ is an almost $AR$-graph. 
\end{proof}

\begin{theorem}
   The complete graph $K_n$ is an $AR$-graph if and only if $n \leq 5$.
\end{theorem}

\begin{proof}
Note that $K_2$ and $K_3$ are trivially $AR$-graphs, and Figure \ref{Complete Graphs} shows that $K_4$ and $K_5$ are also $AR$-graphs.
\begin{figure}[ht]
    \centering
    \includegraphics[scale=0.65]{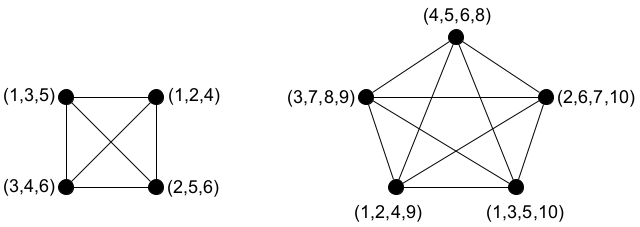}
    \caption{$AR$-Labeling of $K_4$ and $K_5$}
    \label{Complete Graphs}
\end{figure}

Since $K_n$ is $(n-1)$-regular, from Theorem \ref{Delta}, $ARI (K_n ) \geq ES(n-1)$. Also, we have $m(K_n ) = \frac{n^{2} -n}{2}$. For $n$ = 6, $ARI(K_6 ) \geq ES(5) = 13$ and $m(K_6 ) = 15$. To obtain an $AR$-labeling of $K_6$, we need six 5-element sets with each set having distinct subset sums. Since $ES(5) = 13$, each vertex must have an edge of label at least 13 incident on it. If $K_6$ is an $AR$-graph, then the maximum edge label must be 15. Again, since the edge independence number of $K_6$ is 3, three independent edges must be labeled 13, 14 and 15. Therefore, we need two 5-element subsets each with maximum element 13, 14 and 15, respectively, and their pairwise intersection contains only 13, 14 and 15, respectively. But Lunnon showed \cite{Lun} that there exist only two 5-element sets with maximum element 13 having distinct subset sums and their intersection has four elements. Hence, there is no $AR$-labeling of $K_6$ with $ARI(K_6)=15$. Therefore $K_6$ is not an $AR$-graph.\par
Suppose $\{a_1 , a_2, \dots , a_n \}$ is a set of integers having distinct subset sums. Erdős noted \cite{3} that $2^n - 1 \leq nx$, where $x = max\{a_1 , a_2, \dots , a_n \}$. By this inequality, we have $ES(n) \geq \frac{2^n - 1}{n}$, which implies $ES(9) \geq \frac{2^9 - 1}{9}~> 56~>45 = m(K_{10})$. Now, considering the derivatives of $m(K_n)$ and $ES(n)$ with respect to $n$, we can see that $ES(n)$ has an exponential rate of growth compared to $m(K_n)$. So, $ES(n) > m(K_n )$, for all $n \geq 10$. Now, $ES(6) = 24, ES(7) = 44, ES(8) = 84$; $m(K_7 ) = 21, m(K_8 ) = 28, m(K_9 ) = 36$ and $ARI(K_n) > ES(n-1)$. Hence, $K_n$ is not an $AR$-graph, for $n \geq 6$.
\end{proof}

\begin{theorem}
    The only complete bipartite $AR$-graphs are $K_{1,1}, K_{1,2}, K_{2,2}, K_{2,3}, K_{2,4}, K_{3,3},$ $K_{3,4}, K_{4,4}, K_{4,5}$ and $K_{5,5}$.
\end{theorem}
\begin{proof}
It follows from Theorem \ref{Delta} that $K_{m,n}$, $m \leq n$, is not an $AR$-graph, if $ES(n) > mn$. Therefore, it immediately follows that the complete bipartite graphs other than $K_{3,5}, K_{4,6}, K_{5,6}, K_{6,6}$ and those listed in the theorem are not $AR$-graphs. The fact that $K_{3,5}, K_{4,6}, K_{5,6}$ and $K_{6,6}$ are not $AR$-graphs can be verified using the following Python program. The program is based on the fact that for $K_{m,n}$ to be an $AR$-graph, we need at least $m$ distinct $n$-element subsets of $\{1,2,\dots,mn\}$ such that each of them have distinct subset sums. The following program identifies all $n$-element subsets of $\{1,2,\dots,mn\}$ such that each of them have distinct subset sums and check whether there exist $m$ disjoint subsets in this collection.

\begin{lstlisting}[language=Python]
import itertools
print("Enter the value of m and n in K_{m,n} where m < n")

m = int(input("Enter the value of m:"))
n = int(input("Enter the value of n:"))
def get_all_subsets(s):
    """Generates all subsets of a set s."""
    return [list(subset) for subset in itertools.chain.from_iterable(itertools.combinations(s, r) for r in range(len(s) + 1))]

def check_distinct_subset_sums(s):
    """Check if all subset sums are distinct."""
    subsets = get_all_subsets(s)
    sums = {sum(subset) for subset in subsets}
    return len(sums) == len(subsets)  # if sums are distinct, the number of distinct sums should equal the number of subsets

def find_sets_with_distinct_subset_sums():
    """Find all sets of n natural numbers less than mn where the subset sums are distinct."""
    all_numbers = list(range(1, m*n+1))
    possible_sets = itertools.combinations(all_numbers, n)
    
    valid_sets = []
    for s in possible_sets:
        if check_distinct_subset_sums(s):
            valid_sets.append(s)
    
    return valid_sets

def check_disjoint_sets(sets):
    """Check if there are m disjoint sets."""
    for combo in itertools.combinations(sets, m):
        sets_intersection = [set(combo[i]).intersection(set(combo[j])) for i in range(m) for j in range(i + 1, m)]
        if all(len(intersection) == 0 for intersection in sets_intersection):
            return combo
    return None

# Main code to find and check the sets
valid_sets = find_sets_with_distinct_subset_sums()
print(f"Found {len(valid_sets)} valid sets.")

disjoint_sets = check_disjoint_sets(valid_sets)
if disjoint_sets:
    print(f"Found {m} disjoint sets with distinct subset sums:")
    for s in disjoint_sets:
        print(s)
else:
    print(f"No {m} disjoint {n}-element sets found in the given range.")
\end{lstlisting}
\begin{figure}[ht]
    \centering
    \includegraphics[scale=0.50]{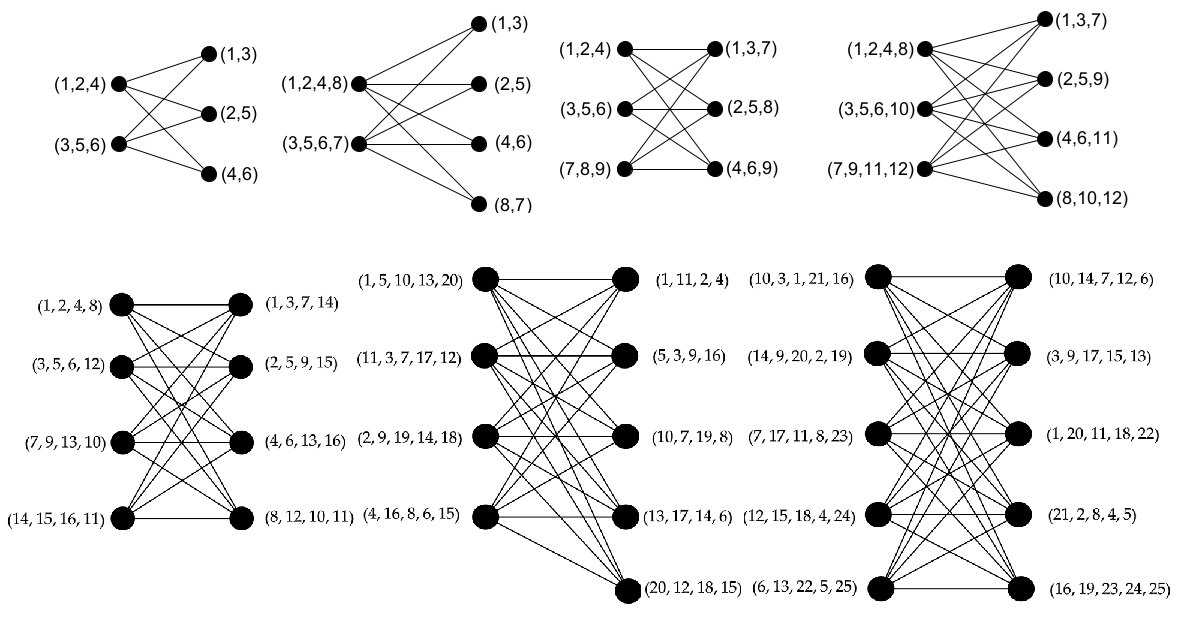}
    \caption{$AR$-Labeling of complete bipartite graphs}
    \label{bipartite}
\end{figure}
\end{proof} 

\begin{theorem}
    The only complete multipartite $AR$-graphs with each partite set having at least two vertices are $K_{2,2,2}$ and $K_{2,2,3}$. 
\end{theorem}
\begin{proof}
   For all complete multipartite graphs $G$ with more than two elements in each partite set, $m(G) < ES(\Delta(G))$, except for $K_{2,2,2}, K_{2,2,3}$ and $K_{3,3,3}$. Figure \ref{Multipartite graphs} shows that $K_{2,2,2}$ and $K_{2,2,3}$ are $AR$-graphs. Now, $\Delta(K_{3,3,3}) = 6$ and $m(K_{3,3,3}) = 27$. Since $ES(6) = 24$,
   on an $AR$-vertex with degree six, the maximum edge label incident should be at least 24. However, there are only four edge labels $x$ with $24\leq x\leq27$, so that at most eight vertices can have maximum edge labels incident on them to be greater than or equal to 24.  But, there are 9 vertices in $K_{3,3,3}$ having degree six. Hence $K_{3,3,3}$ is also not an $AR$-graph.
   \begin{figure}[ht]
    \centering
    \includegraphics[scale=0.365]{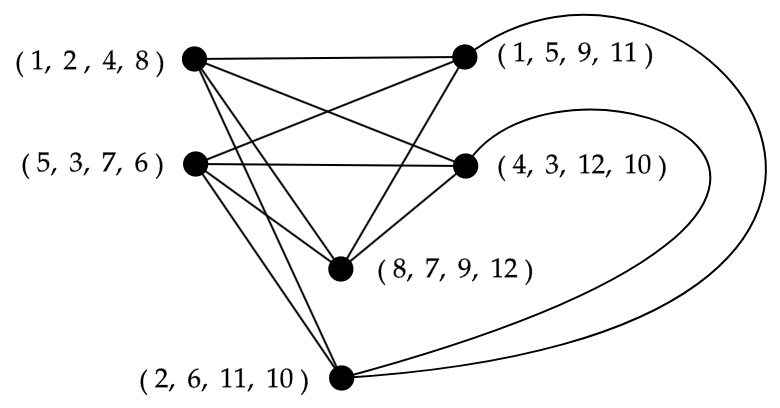}
    \includegraphics[scale=0.362]{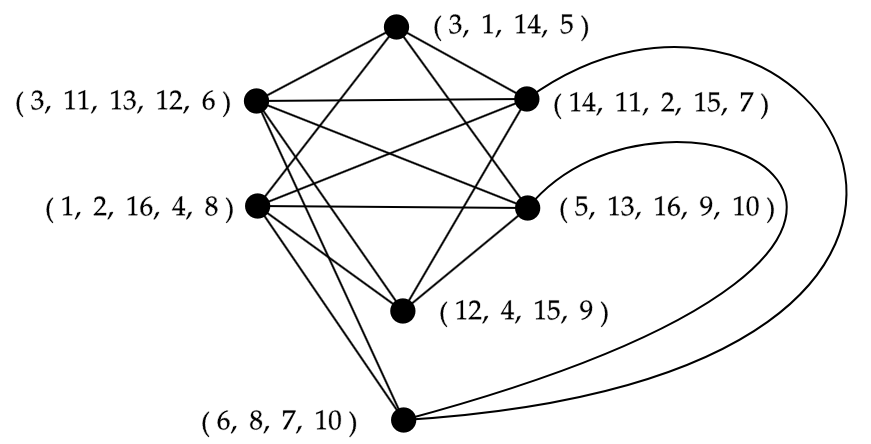}
    \caption{$AR$-Labeling of Multipartite graphs $K_{2,2,2}$ and $K_{2,2,3}$}
    \label{Multipartite graphs}
\end{figure}
\end{proof}    

The following lemma is useful to find the $AR$-index of wheel graphs.

\begin{lemma}\label{3lab} \cite{Apa}
    Given a vertex $v$ in G, if any two edges incident on $v$ are labeled $x$ and $y$, then a third edge can be $z$ if and only if $x + y \neq z$ and $|x - y| \neq z$. Moreover, if the edges incident on a vertex $v$ of degree three are labeled $x$, $y$ and $z$ with $x<y<z$, then $v$ is an AR-vertex if $x + y \neq z$.

\end{lemma}

\begin{theorem}
    The $AR$-index of the wheel graph on $n$ vertices, $ARI$ $(W_n ) = ES(n-1)$, for $n > 5$, where $W_n$ is the cycle $C_{n-1}$ together with a vertex adjacent to all the vertices of $C_{n-1}$. 
\end{theorem}

\begin{proof}
    Since $\Delta(W_n ) = n - 1$, we have $ARI (W_n ) \geq ES(n-1)$. From the $AR$-labeling of $W_6$ and $W_7$ in Figure \ref{W6 and W7}, $ARI (W_6 ) \leq ES(5)$ and $ARI(W_7 ) \leq ES(6)$. Hence, the result is true for $n=6,7$.\par
\begin{figure}[ht]
    \centering
    \includegraphics[scale=0.7]{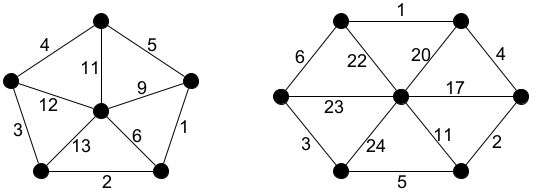}
    \caption{$AR$-Labeling of $W_6$ and $W_7$}
    \label{W6 and W7}
\end{figure}

For $n \geq 8$, label the $n-1$ edges incident on the central vertex using the set of $(n-1)$ elements having distinct subset sums (the set corresponding to $ES(n-1)$). Label a
maximum independent set of edges in the external cycle using some numbers less than $ES(n-1)$) which are not yet used.\par

\noindent{\bf Case 1:} External cycle is even, of length say $j$.\\
From Lemma \ref{3lab}, corresponding to each edge that is left to be labeled, there exists at
most four labels that cannot be used (two each for each vertex). So, in total, there exist $2j$
labels that cannot be used. We have already used up $j + \frac{j}{2}$ labels. If we combine all these
labels that cannot be used for labeling the remaining edges (at least one of them), its still
less than $4j$. But $ES(j) > 6j$ for $j \geq 7$, so there are more than $2j$ labels still left that can be used to label
the remaining $\frac{j}{2}$ edges.\par

\noindent{\bf Case 2:} External cycle is odd, of length $l,~l = 2k + 1$, for some $k \in \mathbb{N}$.\\
After labeling a maximum set of $k$ independent edges, there exists a pair of adjacent edges not yet labeled. Label one of them using some number that is neither used in the labeling yet nor the sum of weights of edges already incident on one of the vertices it falls on. We have used
up $l+k+1$ edge labels and by Lemma \ref{3lab}, the maximum possible numbers that cannot be used in our labeling is $4k$. Their sum amounts to $7k+2$ which is less than $4l$. But, $ES(l) > 6l$ for $l \geq 7$, so there are more than $2l = 4k+2$ labels still left that can be used
to label the remaining edges which are $k$ in number.\\
Combining the restriction forced by $\Delta(W_n )$ and the above constructions, we have for $n >
5$, $ARI (W_n ) = ES(n-1)$.
\end{proof}
\begin{coro*}
The only $AR$-wheels are $W_4$ and $W_5$.    
\end{coro*}
\begin{proof}
    We know that $W_4$ is $K_4$ itself, which we have already shown as an $AR$-graph and $W_5$ is shown to be an $AR$-graph in Figure \ref{W5}.
\begin{figure}[ht]
    \centering
    \includegraphics[scale=0.6]{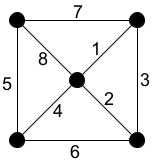}
    \caption{$AR$-Labeling of $W_5$}
    \label{W5}
\end{figure}
For $n > 5$, $W_n$ are not $AR$-graphs since $m(W_n ) = 2(n - 1) < ES(n-1)$.
\end{proof}

\section{Concluding Remarks}
The $AR$-index of a graph $G$ was defined in \cite{Apa} with the intention of measuring how close a graph is from being an $AR$-graph. In this paper, a lower bound for the $AR$ -index of $G$ was obtained by taking into account the maximum degree of $G$ and using a sequence which we call the $ES$-sequence. This bound was used to prove that there are only finitely many $AR$-graphs in some graph classes. This paper outlines the idea of $AR$-index of graphs by giving some bounds that prove handy in identifying some non-$AR$-graphs.\par
The exact values of $AR$-index of majority of the basic graph classes are yet to be determined. The upper bound for $AR$-index is very crude for an arbitrary graph, though two graphs ($K_2$ and $P_3$) do attain it. In fact, though the bounds for $AR$-index are sharp in the general setting, while analyzing particular graphs/classes of graphs, they could be improved drastically. Since $AR$-labeling in general and $AR$-index in particular require elements from $ES$-sequence for labeling as well as improving bounds, it is a necessity moving forward that more elements of $ES$-sequence be computed possibly using some combinatorial and algorithmic techniques. Another question that could possibly arise from the framework of this paper is whether there exists a Ramsey-like result for Erdős subset sum conjecture.\par

\noindent \textbf{Acknowledgment:} The first author is supported by the Junior Research Fellowship (09/0239(17181)/2023-EMR-I) of CSIR (Council of Scientific and Industrial Research, India).

\end{document}